\documentclass[reqno,a4paper, 11pt]{amsart}

\usepackage[a4paper=true,pdfpagelabels]{hyperref}
\usepackage{graphicx}
\usepackage{mathabx}
\usepackage[ansinew]{inputenc}
\usepackage{amsfonts,epsfig}
\usepackage{latexsym}
\usepackage{amsmath}
\usepackage{amssymb}
\usepackage{color}

\newtheorem{theorem}{Theorem}
\newtheorem{lemma}[theorem]{Lemma}

\newtheorem{proposition}[theorem]{Proposition}

\newtheorem{lettertheorem}{Theorem}
\newtheorem{letterlemma}[lettertheorem]{Lemma}

\theoremstyle{definition}

\theoremstyle{remark}

\numberwithin{equation}{section}

\newcommand{\B}{\mathcal{B}}

\newcommand{\D}{\mathbb{D}}
\newcommand{\DD}{\widehat{\mathcal{D}}}
\newcommand{\Dd}{\widecheck{\mathcal{D}}}
\newcommand{\DDD}{\mathcal{D}}
\newcommand{\N}{\mathbb{N}}
\newcommand{\RR}{\mathbb{R}}

\newcommand{\C}{\mathbb{C}}

\newcommand{\e}{\varepsilon}
\newcommand{\ep}{\varepsilon}

\renewcommand{\phi}{\varphi}

\newcommand{\M}{\mathcal{M}}

\def\BMO{\mathord{\rm BMO}}
\def\MO{\mathord{\rm MO}}
\def\BO{\mathord{\rm BO}}
\def\BA{\mathord{\rm BA}}

       \def\b{\beta}        \def\g{\gamma}
           \def\e{\varepsilon}
     \def\om{\omega}      
       \def\t{\theta}       
                  \def\z{\zeta}

\def\omg{{\widehat{\omega}}}

\renewcommand{\H}{\mathcal{H}}

\newenvironment{Prf}{\noindent{\emph{Proof of}}}
{\hfill$\Box$ }

\begin{document}

\title[Small Hankel operator induced by measurable symbol]{Small Hankel operator induced by measurable symbol acting on weighted Bergman spaces}


\thanks{The research reported in this paper is supported in part by Ministerio de Ciencia e Innovaci\'on, Spain, project PID2022-136619NB-I00, La Junta de Andaluc{\'i}a, project FQM210 and the Academy of Finland project no.~356029.}

\author{Jos\'e \'Angel Pel\'aez}
\address{Departamento de An\'alisis Matem\'atico, Universidad de M\'alaga, Campus de
Teatinos, 29071 M\'alaga, Spain} 
\email{japelaez@uma.es}

\author[Jouni R\"atty\"a]{Jouni R\"atty\"a}
\address{University of Eastern Finland\\
Department of Physics and Mathematics\\
P.O.Box 111\\FI-80101 Joensuu\\
Finland}
\email{jouni.rattya@uef.fi}


\date{\today}

\maketitle

\begin{abstract}
The boundedness of the small Hankel operator $h^\omega_{f}(g)=\overline{P_\omega}(fg)$ induced by a measurable symbol $f$ and the Bergman projection $P_\omega$ associated to a radial weight $\omega$ acting from the weighted Bergman space $A^p_\omega$ to its conjugate analytic counterpart $\overline{A^p_\omega}$ is characterized on the range $1<p<\infty$ when $\omega$ belongs to the class $\mathcal{D}$ of radial weights admitting certain two-sided doubling conditions. On the way to the proof a sharp integral estimate for certain modified Bergman kernels is obtained. 
\end{abstract}

\section{Introduction and main results}

Let $\H(\D)$ denote the space of analytic functions in the unit disc $\D$. A non-negative function $\om\in L^1([0,1))$ is called a radial weight if
$\om(z)=\om(|z|)$ for all $z\in\D$. For $0<p<\infty$ and such an $\omega$, the weighted Lebesgue space $L^p_\om$ consists of complex-valued measurable functions $g$ on~$\D$ such that
	$$
	\|g\|_{L^p_\omega}^p=\int_\D|g(z)|^p\omega(z)\,dA(z)<\infty,
	$$
where $dA(z)=\frac{dx\,dy}{\pi}$ is the normalized Lebesgue area measure on $\D$. The corresponding weighted Bergman space is $A^p_\om=L^p_\omega\cap\H(\D)$. Throughout this paper we assume that the tail integral $\widehat{\om}(z)=\int_{|z|}^1\om(s)\,ds$ is strictly positive for all $z\in\D$, for otherwise $A^p_\om=\H(\D)$ due to the radiality of $\om$. As usual,~$A^p_\alpha$ stands for the classical weighted Bergman space induced by the standard radial weight $\omega(z)=(\alpha+1)(1-|z|^2)^\alpha$ with $-1<\alpha<\infty$.

Every radial weight $\om$ induces the orthogonal Bergman projection from $L^2_\om$ to $A^2_\om$ given by 
		\begin{equation*}
    P_\om(g)(z)=\int_{\D}g(\z)\overline{B^\om_{z}(\z)}\,\om(\z)dA(\z),\quad z\in\D,
    \end{equation*}
where $B^\om_{z}$ are the reproducing kernels of the Hilbert space $A^2_\om$. This projection in turn, together with a measurable symbol $f$, induces the small Hankel operator 
	$$
  h^\om_{f}(g)(z)=\overline{P_\om}(fg)(z)=\int_\D f(\z)g(\z)B^\om_z(\z)\om(\z)\,dA(\z),\quad z\in\D.
  $$
The existing literature does not seem to contain much information on this operator unless the symbol is analytic. However, it is known that when $\omega$ is the standard radial weight and $f$ is an admissible measurable function, then $h^\om_{f}=h^\alpha_{f}$ satisfies
	\begin{equation}\label{lhgkjhgkjhg}
	\|h^\alpha_{f}\|_{A^2_\alpha\to\overline{A^2_\alpha}}
	\asymp\sup_{z\in\D}\,(1-|z|)^{2+\alpha}\left|\int_\D f(\z)\frac{(1-|\zeta|)^{\alpha}}{(1-\overline{z}\zeta)^{2(2+\alpha)}}dA(\zeta)\right|
	\end{equation}
by \cite{ZhuTAMS91,Zhu}. Other comments on known results are postponed for a moment.

The main purpose of this paper is to establish a description of the admissible measurable symbols $f$ such that $h^\om_f:A^p_\omega\to\overline{A^p_\om}$
is bounded, provided $1<p<\infty$ and $\omega$ belongs to a relatively wide class of radial weights. To state our main results, some more notation are in order. A radial weight $\om$ belongs to $\DD$ if there exists $C=C(\om)>0$ such that 
	$$
	\omg(r)\le C\omg\left(\frac{1+r}{2}\right),\quad r\to1^-,
	$$
and $\om\in\Dd$ if  
	$$
	\omg(r)\ge C\omg\left(1-\frac{1-r}{K}\right),\quad r\to1^-,
	$$ 
for some $K=K(\om)>1$ and $C=C(\om)>1$. The intersection $\DD\cap\Dd$ is denoted by $\DDD$. It contains the standard radial weights but exponentially decreasing weights do not belong to~$\DD$. The definitions of both $\DD$ and $\Dd$ have their obvious geometric interpretations. Because of these classes are defined by integral conditions, the containment in $\DD$ or $\Dd$ does not require differentiability, continuity or strict positivity. In fact, weights in these classes may vanish on a relatively large part of each outer annulus $\{z:r\le|z|<1\}$ of $\D$. It is known that the classes $\DD$ and $\DDD$ emerge naturally in many instances in operator theory related to Bergman spaces. For example, they are intimately related to Littlewood-Paley estimates, bounded Bergman projections and the Dostani\'c problem~\cite{PelRat2020}. For basic properties and illuminating examples of weights in these and other related classes of weights, see \cite{PelSum14,PR,PR2020,PelRat2020} and the relevant references therein. 

The precise statement of our main result involves the transformation 
	\begin{equation}\label{V}
	V_{\om,\nu}(f)(z)
	=\nu(z)\int_\D f(\zeta)\overline{B^{\om\nu}_z(\zeta)}\om(\zeta)\,dA(\zeta)
	,\quad z\in\D,\quad f\in L^1_\om,
	\end{equation} 
which makes sense for each radial weight $\omega$ and each radial function $\nu:\D\to [0,\infty)$ such that also the product $\nu\omega$ is a weight. The most natural starting point for measurable symbols in our setting is the requirement 
	\begin{equation}\label{fubinicondition}
  \|f\|_{L^1_{\om_{\log}}}=\int_\D|f(z)|\left(\log\frac{e}{1-|z|}\right)\om(z)\,dA(z)<\infty,
	\end{equation}
which we simply denote as $f\in L^1_{\om_{\log}}$. With these preparations we can state the main result of this study.

\begin{theorem}\label{Theorem:hankel-characterization}
Let $1<p<\infty$, $\om\in\DDD$ and $f\in L^1_{\om_{\log}}$. Then $h^\om_{\overline{f}}: A^p_\om\to A^p_\omega$ is bounded if and only if there exists $n_0=n_0(\om)\in\N$ such that, for each $n\ge n_0$, the weight $\nu(z)=(1-|z|)^n$ satisfies $V_{\om,\nu}(f)\in L^\infty$. Moreover, 
	\begin{equation}\label{lweqfjblqfekjbqfwelkjb}
	\|h^\om_{\overline{f}}\|_{A^p_\om\to\overline{A^p_\om}}
	\asymp\|V_{\om,\nu}(f)\|_{L^\infty}
	\end{equation}
for each fixed $n\ge n_0$.
\end{theorem}

As mentioned already, Theorem~\ref{Theorem:hankel-characterization} is known for the standard radial weights when $p=2$ \cite{ZhuTAMS91,Zhu}, and the connection between \eqref{lhgkjhgkjhg} and \eqref{lweqfjblqfekjbqfwelkjb} becomes obvious by choosing $\om$ and $\nu$ to be the standard weights $(1+\alpha)(1-|z|^2)^\alpha$ and $(3+\alpha)(1-|z|^2)^{2+\alpha}$, respectively. Moreover, the question for which analytic symbols $f$ the operator $h^\om_{f}: A^p_\omega\to \overline{A^q_\nu}$ is bounded when $0<p,q<\infty$ and $\omega,\nu\in\DDD$ was recently solved in \cite{DRWW2022}. Therefore the main novelty of Theorem~\ref{Theorem:hankel-characterization} stem from two facts: it concerns small Hankel operators induced by measurable symbols on the range $1<p<\infty$, and it concerns weights in $\DDD$ which is in a sense the largest class of radial weights that contains the standard weights as smooth prototypes. This latter fact forces us deal with the Bergman reproducing kernels $B^\om_z$ induced by $\omega\in\DDD$. Observe that these kernels, on contrary to the standard ones $(1-\overline{\z}z)^{-(2+\alpha)}$, do not have a neat formula in general and may have zeroes \cite{ BoudreauxJGA19,Perala}.

Let us recall that the big Hankel operator and the small Hankel operator are essentially the same on the Hardy spaces in the sense of unitary transformation \cite{Ni} whereas they behave different acting  on  Bergman spaces \cite{Zhu}. In fact while there exist plenty of studies on big Hankel operators only few works deal with the boundedness of the small Hankel operators on Bergman spaces \cite{Aleman:Constantin2004,BDTW2021,BL2005, DRWW2022,JPR87,Oliver,Yamaji,ZhuTAMS91}.  Among this non-exhaustive but representative collection, only the works by Janson, Petree and Rochberg \cite{JPR87}, Yamaji \cite{Yamaji} and Zhu \cite{ZhuTAMS91} concern small Hankel operators induced by measurable symbols, and none of these ones involve kernels induced by weights in such a class as $\DDD$. 

As for the proof of Theorem~\ref{Theorem:hankel-characterization},
it comes by no means for free and actually require a number of technically demanding steps. Indeed, for the necessity part we will use the kernel estimate
		\begin{equation}\label{eq:kernelmix-general'}
    \int_\D|(1-\overline{z}\z)^NB^\om_z(\z)|^p\nu(\z)\,dA(\z)
		\le C\left(\int_0^{|z|}\frac{\widehat{\nu}(t)}{\widehat{\om}(t)^p(1-t)^{p(1-N)}}\,dt+1\right),\quad z\in\D,
    \end{equation}
valid for $\om\in\DD$, $N\in\N$ and $2\le p<\infty$. Section~\ref{s2} is dedicated to the proof of this result, and the result itself is stated as Lemma~\ref{kernelmix-general}. The case $N=0$ in \eqref{eq:kernelmix-general'} was established in \cite[Theorem~1]{PR2016/1}, and has turned out to be a premordial tool in several question related to the study of concrete operators on weighted Bergman spaces \cite{DRWW2022,KR,PPR2020,PR2016/1,PelRat2020,PRMZ2023}.
 
In the case of analytic symbols it is known that $h^\om_{f}:A^p_\omega\to\overline{A^p_\om}$ is bounded for some (equivalently for all) $1<p<\infty$ if and only if $f$ belongs to the classical Bloch space $\B$ by \cite[Theorem~3]{DRWW2022}. The proof there uses Lemma~\ref{kernelmix-general}, and it is therefore not independent of our study. For the sake of simplicity and completeness, we will provide a direct proof of the aforementioned result on small Hankel operators as a byproduct of Theorem~\ref{Theorem:hankel-characterization}. This allows one to avoid some of the heavy machinery employed in the proof \cite[Theorem~3]{DRWW2022}. Indeed, we will show in Proposition~\ref{pr:Blochdescription} below that $\|f\|_{\B}\asymp\|V_{\om,\nu}(f)\|_{L^\infty}$ with $\nu(z)=(1-|z|)^n$. Therefore, in the case of anti-analytic symbols, the following neat result confirms the fact that in many instances the conditions that characterize the bounded Hankel operators (small or big) and the integration operator, defined by 
	$$
	T_g(f)(z)=\int_0^zg'(\zeta)f(\zeta)\,d\zeta,\quad z\in\D,
	$$
on weighted Bergman spaces with $1<p<\infty$ are the same~\cite{AlCo,DRWW2022,KR,PelSum14,PR,Zhu}. 

\begin{theorem}\label{proposition:Bloch-necessary}
Let $1<p<\infty$, $\om\in\DDD$ and $f\in\H(\D)$. Then $h^\om_{\overline{f}}:A^p_\om\to\overline{A^p_\om}$ is bounded if and only if $f\in\B$. Moreover,
	$$
	\|h^\om_{\overline{f}}\|_{A^p_\om\to\overline{A^p_\om}}\asymp\|f\|_{\B}.
	$$
\end{theorem}

We also observe that the statement in Theorem~\ref{proposition:Bloch-necessary} fails for $0<p\le1$. In the case $p=1$ the characterizing condition is
	\begin{equation}\label{msufk}
  \sup_{z\in\D}\left|f'(z)\right|(1-|z|)\log\frac{e}{1-|z|}<\infty,
  \end{equation}
while for $0<p<1$ we have
	$$
	\|h^\om_{\overline{f}}\|_{A^p_\om\to\overline{A^p_\om}}
	\asymp\sup_{z\in\D}\frac{|f'(z)|(1-|z|)}{\left(\widehat{\om}(z)(1-|z|)\right)^{\frac{1}{p}-1}}
	$$
by \cite[Theorems~5 and 7]{DRWW2022}.

The remainder of the paper is organized as follows. The estimate \eqref{eq:kernelmix-general'} together with some preliminary results on weights are proved in Section~\ref{s2} while Section~\ref{s3} is devoted to the proofs of Theorems~\ref{Theorem:hankel-characterization} and \ref{proposition:Bloch-necessary}. At the end of the paper in Section~\ref{s4} we will use the results obtained to sketch a new proof of the main result of the recent study \cite{PRMZ2023} related to the boundedness of $P_\om:L^\infty\to\B$ when $\om\in\DDD$. 

We finish this introduction with couple of words on the notation already used. The letter $C=C(\cdot)$ will denote an absolute constant whose value depends on the parameters indicated in the parenthesis, and may change from one occurrence to another. If there exists a constant
$C=C(\cdot)>0$ such that $a\le Cb$, we will write either $a\lesssim b$ or $b\gtrsim a$. In particular, if $a\lesssim b\lesssim a$, then we write $a\asymp b$ and say that $a$ and $b$ are comparable.

\section{Norm estimates for modified Bergman kernels}\label{s2}

The aim of this section is to establish a sharp estimate for the norm of the functions $\zeta\mapsto(1-\overline{z}\zeta)^NB^\om_z(\zeta)$ in $A^p_\nu$. This norm estimate plays a crucial role in our arguments to prove the main results presented in the introduction. We begin with known characterizations of weights in $\DD$. The next result follows by \cite[Lemma~2.1]{PelSum14} and its proof. 

\begin{letterlemma}\label{Lemma:weights-in-D-hat}
Let $\om$ be a radial weight. Then the following statements are equivalent:
\begin{itemize}
\item[\rm(i)] $\om\in\DD$;
\item[\rm(ii)] There exist $C=C(\om)>0$ and $\b=\b(\om)>0$ such that
    \begin{equation*}
    \begin{split}
    \widehat{\om}(r)\le C\left(\frac{1-r}{1-t}\right)^{\b}\widehat{\om}(t),\quad 0\le r\le t<1;
    \end{split}
    \end{equation*}
\item[\rm(iii)] There exists $\lambda=\lambda(\om)\ge0$ such that
    $$
    \int_\D\frac{\om(z)}{|1-\overline{\z}z|^{\lambda+1}}\,dA(z)\asymp\frac{\widehat{\om}(\zeta)}{(1-|\z|)^\lambda},\quad \z\in\D;
    $$
\item[\rm(iv)] There exist $C=C(\om)>0$ and $\eta=\eta(\om)>0$ such that
    \begin{equation*}
    \begin{split}
    \om_x\le C\left(\frac{y}{x}\right)^{\eta}\om_y,\quad 0<x\le
    y<\infty.
    \end{split}
    \end{equation*}
\end{itemize}
\end{letterlemma}

The following simple lemma is useful for our purposes. The second case implies that $\DDD$ is closed under multiplication by the hat of any weight in $\DD$. The third case reveals the same phenomenon when the multiplication is done by a suitably small negative power of the hat.  For each radial weight $\om$ and $\b\in\RR$, write $\om_{[\b]}(z)=\om(z)(1-|z|)^\b$ for all $z\in\D$.

\begin{lemma}\label{room}
Let $\nu$ be a radial weight. Then the following statements are equivalent: 
\begin{itemize}
\item[\rm(i)] $\nu\in\DD$;
\item[\rm(ii)] For some (equivalently for each) $\om\in\DDD$ we have $\omega\widehat{\nu}\in\DDD$ and $\widehat{\omega\widehat{\nu}}\asymp\widehat{\om}\widehat{\nu}$ on $[0,1)$;
\item[\rm(iii)] For some (equivalently for each) $\om\in\DDD$ there exists $\gamma_0=\gamma_0(\om,\nu)>0$ such that, for each $\g\in(0,\g_0]$, we have $(\widehat{\nu})^{-\gamma}\om\in\DDD$, and 
	\begin{equation}\label{chochi}
	\int_r^1\frac{\om(t)}{\widehat{\nu}(t)^{\gamma}}\,dt\asymp\frac{\widehat{\om}(r)}{\widehat{\nu}(r)^{\gamma}},\quad 0\le r<1.
	\end{equation}
\end{itemize}
\end{lemma}

\begin{proof}
Assume first (i), and let $\om\in\DDD$. It is clear that $\widehat{\omega\widehat{\nu}}\le\widehat{\om}\widehat{\nu}$ on $[0,1)$. Let $\beta=\beta(\nu)>0$ be that of Lemma~\ref{Lemma:weights-in-D-hat}(ii), and observe that $\om\in\Dd$ if
	\begin{equation}\label{Eq:d-check}
	\widehat{\om}(r)\le C'\int_r^{1-\frac{1-r}{K}}\om(t)\,dt,\quad0\le r<1,
	\end{equation}
for some constants $C'=C'(\om)>1$ and $K=K(\om)>1$ by the definition. Then Lemma~\ref{Lemma:weights-in-D-hat}(ii) together with \eqref{Eq:d-check} yields
	\begin{equation*}
	\begin{split}
	\widehat{\om\widehat{\nu}}(r)
	&\ge\frac{\widehat{\nu}(r)}{C(1-r)^\beta}\int_r^1\omega_{[\beta]}(t)\,dt
	\ge\frac{\widehat{\nu}(r)}{C(1-r)^\beta}\int_{r}^{1-\frac{1-r}{K}}\omega_{[\beta]}(t)\,dt\\
	&\ge\frac{\widehat{\nu}(r)}{C(1-r)^\beta}\frac{(1-r)^\beta}{K^\beta}\int_{r}^{1-\frac{1-r}{K}}\omega(t)\,dt
	\ge\frac{\widehat{\nu}(r)\widehat{\om}(r)}{CC'K^\beta}, \quad 0\le r<1,
	\end{split}
	\end{equation*}
and thus $\widehat{\omega\widehat{\nu}}\asymp\widehat{\om}\widehat{\nu}$ on $[0,1)$. Standard arguments involving this asymptotic equality, the definition of $\DD$ and \eqref{Eq:d-check} show that $\om\widehat{\nu}\in\DDD$. Thus (ii) is satisfied.

Assume next (ii). Then $\om\widehat\nu\in\DD$ implies 
	\begin{equation*}
	\begin{split}
	\widehat{\om}(r)\widehat{\nu}(r)
	&\lesssim\widehat{\om\widehat{\nu}}(r)
	\lesssim\widehat{\om\widehat{\nu}}\left(\frac{1+r}{2}\right)
	\le\widehat{\nu}\left(\frac{1+r}{2}\right)\widehat\om(r),\quad 0\le r<1,
	\end{split}
	\end{equation*}
and hence $\nu\in\DD$. Thus (i) and (ii) are equivalent.

The fact that (i) implies (iii) follows from \cite[Lemma~2]{PRMZ2023}. Conversely, if (iii) is satisfied, then
	\begin{equation*}
	\begin{split}
	\frac{\widehat{\om}(r)}{\widehat{\nu}(r)^\gamma}
	&\gtrsim\int_{\frac{1+r}{2}}^1\frac{\om(t)}{\widehat{\nu}(t)^\gamma}\,dt
	\ge\frac{\widehat{\om}\left(\frac{1+r}{2}\right)}{\widehat{\nu}\left(\frac{1+r}{2}\right)^\gamma},\quad 0\le r<1,
	\end{split}
	\end{equation*}
and since $\om\in\DD$ and $\gamma>0$, we deduce $\nu\in\DD$.
\end{proof}

The following result of technical nature can be established by arguments similar to those used to prove \cite[(22)]{PR2016/1}. We leave the details of the proof for an interested reader.

\begin{lemma}\label{le:HLextended}
Let $\omega\in\DD$, $0<p<\infty$ and $\alpha\in\mathbb{R}$. Then
	$$
	\sum_{n=0}^\infty \frac{(n+1)^{\alpha-2}}{\omega^p_{2n+1}}s^n
	\asymp\int_0^s \frac{dt}{\widehat{\omega}(t)^p(1-t)^\alpha}+1,\quad 0<s<1.	
	$$
\end{lemma}

In the forthcoming arguments we will face weights which are of the same nature as $\om_{[\beta]}$ but slightly different. Since we have to keep track with the details here carefully, we introduce suitable notation for these creatures. For this purpose, we denote
	$$ 
	\left( \omega_{(n_1,y_1),(n_2,y_2),\dots,(n_k,y_k)}\right)_x
	=\int_{0}^1 r^x\left(\prod_{j=1}^k(1-r^{y_j})^{n_j}\right)\om(r)\,dr,\quad 0\le x<\infty,
	$$
for $k,n_j,y_j\in\N$ and $j\in\{1,\dots,k\}$. It follows from the inequality $1-r^a\le a(1-r)$, valid for all $1\le a<\infty$ and $0\le r\le1$, and \cite[(1.3)]{PelRat2020} that there exists a constant $C=C\left(\omega, \sum_{j=1}^k n_j\right)>0$ such that
	\begin{equation}
	\begin{split}\label{eq:momentogeneral}
	\left(\omega_{(n_1,y_1),(n_2,y_2),\dots,(n_k,y_k)}\right)_x
	&\le\left(\prod_{j=1}^k y_j^{n_j}\right)\left( \omega_{\left[\sum_{j=1}^k n_j \right]}\right)_x\\
	&\le C\left(\prod_{j=1}^k y_j^{n_j}\right)\frac{\omega_x}{x^{\sum_{j=1}^k n_j}}, \quad 0\le x<\infty.
	\end{split}
	\end{equation}

The kernel estimate that we are after reads as follows.

\begin{lemma}\label{kernelmix-general}
Let $2\le p<\infty$, $N\in\N$ and $\om\in\DD$, and let $\nu$ be a radial weight. Then there exists a constant $C=C(\om,\nu,p,N)>0$ such that
    \begin{equation}\label{eq:kernelmix-general}
    \int_\D|(1-\overline{z}\z)^NB^\om_z(\z)|^p\nu(\z)\,dA(\z)
		\le C\left(\int_0^{|z|}\frac{\widehat{\nu}(t)}{\widehat{\om}(t)^p(1-t)^{p(1-N)}}\,dt+1\right),\quad z\in\D.
    \end{equation}
Moreover, if $\nu\in\DDD$, then
		\begin{equation}\label{eq:kernelmix-generalX}
    \int_\D|(1-\overline{z}\z)^NB^\om_z(\z)|^p\nu(\z)\,dA(\z)
		\asymp\left(\int_0^{|z|}\frac{\widehat{\nu}(t)}{\widehat{\om}(t)^p(1-t)^{p(1-N)}}\,dt+1\right),\quad z\in\D.
    \end{equation}
\end{lemma}
		
\begin{proof}
First we will deal with the case $N=1$. A direct calculation shows that
    \begin{equation*}
    \begin{split}
    2(1-\overline{z}\z)B^\om_z(\z)
    &=(1-\overline{z}\z)\sum_{n=0}^\infty\frac{(\overline{z}\z)^n}{\om_{2n+1}}
    =\sum_{n=0}^\infty\frac{(\overline{z}\z)^n}{\om_{2n+1}}
		-\sum_{n=0}^\infty\frac{(\overline{z}\z)^{n+1}}{\om_{2n+1}}\\
    &=\frac1{\om_1}+\sum_{n=1}^\infty\left(\frac{\om_{2n-1}-\om_{2n+1}}{\om_{2n+1}\om_{2n-1}}\right)(\overline{z}\z)^n,
    \end{split}
    \end{equation*}
where
    \begin{equation}\label{eq:difmoments}
    \om_{2n-1}-\om_{2n+1}=\int_0^1r^{2n-1}\om(r)(1-r^2)\,dr=(\om_{(1,2)})_{2n-1},\quad n\in\N.
    \end{equation}
Thus
    \begin{equation}\label{eq:n=1}
    \begin{split}
    2(1-\overline{z}\z)B^\om_z(\z)
		&=\frac1{\om_1}+\sum_{n=1}^\infty\left(\frac{(\om_{(1,2)})_{2n-1}}{\om_{2n+1}\om_{2n-1}}\right)(\overline{z}\z)^n\\
		&=A_1(\omega)+\sum_{n=1}^\infty\left(\frac{(\om_{(1,2)})_{2n-1}}{\om_{2n+1}\om_{2n-1}}\right)(\overline{z}\z)^n ,\quad z,\z\in\D.
    \end{split}
    \end{equation}
By the Hardy-Littlewood inequality \cite[Theorem~6.3]{D} and \eqref{eq:momentogeneral}, we deduce
    \begin{equation*}\begin{split}
    \int_0^{2\pi}|(1-\overline{z}re^{i\t})B^\om_z(re^{i\t})|^p\,d\t
    & \lesssim A^p_1(\omega)+\sum_{n=1}^\infty\left(\frac{(\om_{(1,2)})_{2n-1}}{\om_{2n+1}\om_{2n-1}}\right)^p
		n^{p-2}|\z|^{np}|z|^{np}\\
		&\lesssim A^p_1(\omega)+\sum_{n=1}^\infty\left(\frac{1}{\om_{2n+1}}\right)^pn^{-2}r^{np}|z|^{np},
		\quad z\in \D,\quad 0\le r<1,
    \end{split}
    \end{equation*}
where the sum on the right-hand side corresponds to the case $\alpha=0$ of Lemma~\ref{le:HLextended}. Therefore we obtain
    $$
    \int_0^{2\pi}|(1-\overline{z}re^{i\t})B^\om_z(re^{i\t})|^p\,d\t
    \lesssim A^p_1(\omega)+\int_0^{|z|r}\frac{dt}{\widehat{\om}(t)^p},\quad z\in\D,\quad0\le r<1,
    $$
and it follows that
    \begin{equation}\label{pilipalipuli}
    \begin{split}
    \int_\D|(1-\overline{z}\z)B^\om_z(\z)|^p\nu(\z)\,dA(\z)
    &\lesssim\int_0^1\left(\int_0^{|z|r}\frac{dt}{\widehat{\om}(t)^p}\right)\nu(r)\,dr+1\\
		&=\int_0^{|z|}\frac{1}{\widehat{\om}(t)^p}\left(\int_{t/|z|}^1\nu(r)\,dr\right)\,dt+1\\
    &\le\int_0^{|z|}\frac{\widehat{\nu}(t)}{\widehat{\om}(t)^p}\,dt+1,\quad z\in\D.
    \end{split}
    \end{equation}
Hence the case $N=1$ is proved.

Next we could proceed by induction, but before doing that we discuss the case $N=2$ for the sake of clarity. By using \eqref{eq:n=1}, we obtain
		\begin{equation*}
    \begin{split}
		2(1-\overline{z}\z)^2B^\om_z(\z)
		&=(1-\overline{z}\z)\left(A_1(\omega)+\sum_{n=1}^\infty\frac{(\om_{(1,2)})_{2n-1} (\overline{z}\z)^n}{\om_{2n+1}\om_{2n-1}}\right)\\
		&=A_2(\omega,\overline{z}\z)+\sum_{n=2}^\infty
		\left(\frac{(\om_{(1,2)})_{2n-1}}{\om_{2n+1}\om_{2n-1}}-\frac{(\om_{(1,2)})_{2n-3}}{\om_{2n-1}\om_{2n-3}}\right)(\overline{z}\z)^n,\quad z,\z\in\D,
    \end{split}
    \end{equation*}
where  
	$$
	A_2(\omega,\overline{z}\z)=\frac{1-\overline{z}\z}{\omega_1}+\frac{(\om_{(1,2)})_{1}}{\omega_3\omega_1}\overline{z}\z,\quad z,\z\in\D.
	$$
Obviously, $|A_2(\omega,\overline{z}\z)|\le A_2(\omega)<\infty$ for all $z,\z\in\D$, and therefore it is a harmless term. Next, a reasoning similar to that in \eqref{eq:difmoments} gives
	\begin{equation}
	\begin{split}\label{eq:difmoments2}
	\left(\frac{(\om_{(1,2)})_{2n-1}}{\om_{2n+1}\om_{2n-1}}-\frac{(\om_{(1,2)})_{2n-3}}{\om_{2n-1}\om_{2n-3}}\right)
	&=\frac{1}{\om_{2n-1}}\left(\frac{(\om_{(1,2)})_{2n-1}}{\om_{2n+1}}-\frac{(\om_{(1,2)})_{2n-3}}{\om_{2n-3}}\right)\\
	&=\frac{1}{\om_{2n-1}}\bigg(\frac{(\om_{(1,2)})_{2n-1}-(\om_{(1,2)})_{2n-3}}{\om_{2n+1}}\\
	&\quad+(\om_{(1,2)})_{2n-3} \left(\frac{1}{\om_{2n+1}}-\frac{1}{\om_{2n-3}}\right)\bigg)\\
	&=\frac{1}{\om_{2n-1}}\left(-\frac{(\om_{(2,2)})_{2n-3}}{\om_{2n+1}}+(\om_{(1,2)})_{2n-3}
	\frac{\om_{2n-3}-\om_{2n+1}}{\om_{2n+1}\om_{2n-3}}\right)\\
	&=\frac{1}{\om_{2n-1}}\left(-\frac{(\om_{(2,2)})_{2n-3}}{\om_{2n+1}}
	+\frac{(\om_{(1,2)})_{2n-3}(\om_{(1,4)})_{2n-3}}{\om_{2n+1}\om_{2n-3}}\right)
	\end{split}
	\end{equation}    
for all $n\ge2$. Therefore
	\begin{equation}\label{eq:n=2}
  \begin{split}
	2(1-\overline{z}\z)^2B^\om_z(\z)
	=A_2(\omega,\overline{z}\z)-I(\omega,\overline{z}\z)+II(\omega,\overline{z}\z),\quad z,\z\in\D,
    \end{split}
    \end{equation}
where 
	$$
	I(\omega,\overline{z}\z)
	=\sum_{n=2}^\infty \frac{(\om_{(2,2)})_{2n-3}}{\om_{2n+1}\om_{2n-1}}(\overline{z}\z)^n,\quad z,\z\in\D,
	$$
and 
	$$
	II(\omega,\overline{z}\z)
	=\sum_{n=2}^\infty \frac{(\om_{(1,2)})_{2n-3}(\om_{(1,4)})_{2n-3}}{\om_{2n+1}\om_{2n-1}\om_{2n-3}}(\overline{z}\z)^n,
   \quad z,\z\in\D.
	$$
Next, by \eqref{eq:momentogeneral} and Lemma~\ref{Lemma:weights-in-D-hat}, we deduce
	\begin{equation*}
  \begin{split} 
	\frac{(\om_{(2,2)})_{2n-3}}{\om_{2n+1}\om_{2n-1}}
	\le C(\om,2)\frac{\om_{2n-3}}{(2n-3)^2\om_{2n+1}\om_{2n-1}}
	\lesssim\frac{1}{n^2\om_{2n+1}},\quad n\ge 2,
  \end{split}
  \end{equation*}
and 
  \begin{equation*}
	\begin{split} 
	\frac{(\om_{(1,2)})_{2n-3}(\om_{(1,4)})_{2n-3}}{\om_{2n+1}\om_{2n-1}\om_{2n-3}}
	\le C(\om,2)\frac{\omega_{2n-3}}{(2n-3)^2\om_{2n+1}\om_{2n-1}}
	\lesssim\frac{1}{n^2\om_{2n+1}},\quad n\ge 2.
  \end{split}
  \end{equation*}
Observe that even if $I$ and $II$ obey the same upper estimate, there is no significant cancellation in the difference $I-II$, and therefore we may consider them separately.

The Hardy-Littlewood inequality \cite[Theorem~6.3]{D} now yields
    \begin{equation*}
		\begin{split}
    \int_0^{2\pi}|(1-\overline{z}re^{i\t})^2B^\om_z(re^{i\t})|^p\,d\t
		\lesssim A^p_2(\omega)+\sum_{n=1}^\infty\left(\frac{1}{n\om_{2n+1}}\right)^pn^{-2}r^{np}|z|^{np},\quad z\in \D,\quad 0\le r<1,
    \end{split}
    \end{equation*}
where the sum on the right-hand side corresponds to the case $\alpha=-p$ of
Lemma~\ref{le:HLextended}. Therefore
    $$
    \int_0^{2\pi}|(1-\overline{z}re^{i\t})^2B^\om_z(re^{i\t})|^p\,d\t
    \lesssim A^p_2(\omega)+\int_0^{|z|r}\frac{dt}{\widehat{\om}(t)^p(1-t)^{-p}},\quad z\in \D,\quad 0\le r<1.
    $$
By arguing as in \eqref{pilipalipuli}, we obtain \eqref{eq:kernelmix-general} for $N=2$.    

Let us now proceed by induction. Let $N\in\N\setminus\{1\}$, and assume that there exist $M=M(N)\in\N$ and $L=L(N)\in\N$ such that
	\begin{equation}\label{eq:n=N}
  \begin{split}
	& 2(1-\overline{z}\z)^NB^\om_z(\z)
	=A_N(\omega,\overline{z}\z)+\sum_{j=1}^{M}\delta_jB^\omega_{N,\ell_j,z}(\zeta) \quad z,\z\in\D,
	\end{split}
  \end{equation} 
where $\delta_j\in\{\pm1\}$, $\sup_{z,\z\in\D}|A_N(\omega,\overline{z}\z)|\le A_N(\omega)<\infty$, and 
    \begin{equation*}
    \begin{split}
    B^\omega_{N,\ell_j,z}(\zeta)
		=\sum_{k=N}^\infty
		\frac{\prod_{s\in\{m_1,\dots,m_{\ell_j}\}}\left(\omega_{(n^s_1,y^s_1),(n^s_2,y^s_2),\dots,(n^s_k,y^s_k)}\right)_{2k+1-2s}}
		{\prod_{m=0}^{\ell_j}\om_{2k+1-2m}}(\overline{z}\z)^k,\quad z,\z\in\D,
    \end{split}
    \end{equation*}   
for all $j\in\{1,\dots,M\}$, and
  \begin{equation}
	\begin{split}\label{LN}
	\max_{j=1,\dots,M}\ell_j\le N+1, \quad 
	\max_{s\in\{m_1,\dots,m_{\ell_j}\},\,j=1\dots,M}y^{s}_j\le L,\quad\text{and}\quad
  \sum_{j=1}^N\sum_{s\in\{m_1,\dots,m_{\ell_j}\}}n^{s}_j=N.
  \end{split}
	\end{equation}
 
We observe first that
		\begin{equation}\label{eq:N+1s1}
    \begin{split}
		(1-\overline{z}\z)B^\omega_{N,\ell,z}(\zeta)
		=A^{\ell}_{N+1}(\omega,\overline{z}\z)+\sum_{k=N+1}^\infty \widehat{B^\omega_{N+1,\ell}}(k)(\overline{z}\z)^k,\quad z,\z\in\D,
    \end{split}
    \end{equation}
where $\ell=\ell_j$ with $j\in\{1,\dots,M\}$, $\sup_{z,\z\in\D}|A^{\ell}_{N+1}(\omega,\overline{z}\z)|\le A^{\ell}_N(\omega)<\infty$ and    
		\begin{equation*}\label{eq:N+1s11}
		\begin{split}
    \widehat{B^\omega_{N+1,\ell}}(k)
		&=\frac{\prod_{s\in\{m_1,\dots,m_{\ell}\}}\left(\omega_{(n^s_1,y^s_1),(n^s_2,y^s_2),\dots,(n^s_k,y^s_k)}\right)_{2k+1-2s}}
		{\prod_{m=0}^{\ell}\om_{2k+1-2m}}\\
		&\quad-\frac{\prod_{s\in\{m_1,\dots,m_{\ell}\}}\left(\omega_{(n^s_1,y^s_1),(n^s_2,y^s_2),\dots,(n^s_k,y^s_k)}\right)_{2k-1-2s}}
		{\prod_{m=0}^{\ell}\om_{2k-1-2m}}. 
    \end{split}
		\end{equation*}
Next, by generalizing the proof of \eqref{eq:difmoments2}, we obtain, for each $k\ge N+1$, the identity
		\begin{equation}\label{eq:N+1s2}
    \begin{split}
		\widehat{B^\omega_{N+1,\ell}}(k)\prod_{m=1}^{\ell}\om_{2k+1-2m}
		&=\frac{\prod_{s\in\{m_1,\dots,m_{\ell}\}}\left(\omega_{(n^s_1,y^s_1),(n^s_2,y^s_2),\dots,(n^s_k,y^s_k)}\right)_{2k+1-2s}}{ \om_{2k+1}}\\
		&\quad-\frac{\prod_{s\in\{m_1,\dots, m_{\ell}\}}\left(\omega_{(n^s_1,y^s_1),(n^s_2,y^s_2),\dots,(n^s_k,y^s_k)}\right)_{2k-1-2s}}
		{\om_{2k-1-2\ell}}\\
		&=\frac{1}{\om_{2k+1}}\bigg(\prod_{s\in\{m_1,\dots, m_{\ell}\}}\left(\omega_{(n^s_1,y^s_1),(n^s_2,y^s_2),\dots,(n^s_k,y^s_k)}\right)_{2k+1-2s}\\
		&\quad-\prod_{s\in\{m_1,\dots, m_{\ell}\}}\left(\omega_{(n^s_1,y^s_1),(n^s_2,y^s_2),\dots,(n^s_k,y^s_k)}\right)_{2k-1-2s}\bigg)\\
		&\quad+\frac{\left(\omega_{(1,2\ell)}\right)_{2k+1-2\ell}}{\om_{2k+1}\om_{2k+1-2\ell}} 
		\prod_{s\in\{m_1,\dots,m_{\ell}\}}\left(\omega_{(n^s_1,y^s_1),(n^s_2,y^s_2),\dots,(n^s_k,y^s_k)}\right)_{2k-1-2s}\\
		&=-\frac{1}{\om_{2k+1}}\sum_{j\in\{1,\dots,\ell\}}D_j\\
		&\quad+\frac{\left(\omega_{(1,2\ell)}\right)_{2k+1-2\ell}}{\om_{2k+1}\om_{2k+1-2\ell}}
		\prod_{s\in\{m_1,\dots,m_{\ell}\}}\left(\omega_{(n^s_1,y^s_1),(n^s_2,y^s_2),\dots,(n^s_k,y^s_k)}\right)_{2k-1-2s},
    \end{split}
    \end{equation}
where
    \begin{equation*}
		\begin{split}
    D_j&=\left(\omega_{(1,2),(n^j_1,y^j_1),(n^j_2,y^j_2),\dots,(n^j_k,y^j_k)}\right)_{2k-1-2j}
		\prod_{s\in\{m_1,\dots,m_{j-1}\}}\left( \omega_{(n^s_1,y^s_1),(n^s_2,y^s_2),\dots,(n^s_k,y^s_k)}\right)_{2k-1-2s}\\
		&\quad\cdot\prod_{s\in\{m_{j+1},\dots,m_{\ell}\}}\left( \omega_{(n^s_1,y^s_1),(n^s_2,y^s_2),\dots,(n^s_k,y^s_k)}\right)_{2k+1-2s}.
		\end{split}
		\end{equation*}
In this last identity, a product without factors is considered equal to one. By combining \eqref{eq:N+1s1} and \eqref{eq:N+1s2} we obtain a formula of the type \eqref{eq:n=N} with $N+1$ in place of $N$. Therefore we have now proved \eqref{eq:n=N} for all $N\in\N$. Hence it follows that to get \eqref{eq:kernelmix-general}, it is enough to show that
		\begin{equation}\label{eq:kernelmix-generalell}
    \int_\D|B^\om_{N,\ell,z}(\z)|^p\nu(\z)\,dA(\z)
		\lesssim\int_0^{|z|}\frac{\widehat{\nu}(t)}{\widehat{\om}(t)^p(1-t)^{p(1-N)}}\,dt+1,\quad z\in\D,
    \end{equation}
for each $\ell=\ell_j$ with $j\in\{1,\dots,M\}$. But \eqref{eq:momentogeneral}, \eqref{LN} and Lemma~\ref{Lemma:weights-in-D-hat} yield
		\begin{equation}\label{eq:N+1s3}
    \begin{split}
		\frac{\prod_{s\in\{m_1,\dots,m_{\ell}\}}\left(\omega_{(n^s_1,y^s_1),(n^s_2,y^s_2),\dots,(n^s_k,y^s_k)}\right)_{2k+1-2s}}
		{\prod_{m=0}^{\ell}\om_{2k+1-2m}}
		\lesssim\frac{1}{\omega_{2k+1}k^{N}}.
		\end{split}
		\end{equation}
Putting this together with \eqref{eq:N+1s1} and the Hardy-Littlewood inequality \cite[Theorem~6.3]{D}, we obtain
    \begin{equation*}\begin{split}
    \int_0^{2\pi}|B^\om_{N,\ell,z}(re^{i\t})|^p\,d\t
    \lesssim (A^{\ell}_{N}(\omega))^p
		+\sum_{n=1}^\infty\left(\frac{1}{n^{N-1}\om_{2n+1}}\right)^pn^{-2}r^{np}|z|^{np},\quad z\in\D,\quad0\le r<1,
    \end{split}
    \end{equation*}
where the right-hand side corresponds to  the case $\alpha=(1-N)p$ of Lemma~\ref{le:HLextended}. Therefore we deduce
    $$
    \int_0^{2\pi}|B^\om_{N,\ell,z}(re^{i\t})|^p\,d\t
    \lesssim (A^{\ell}_{N}(\omega))^p + \int_0^{|z|r}\frac{dt}{\widehat{\om}(t)^p(1-t)^{(1-N)p}},
		\quad z\in \D,\quad 0\le r<1,
    $$  
and it follows that
    \begin{equation*}
    \begin{split}
    \int_\D|B^\om_{N,\ell,z}(\z)|^p\nu(\z)\,dA(\z)
    &\lesssim\int_0^1\left(\int_0^{|z|r}\frac{dt}{\widehat{\om}(t)^p(1-t)^{(1-N)p}}\right)\nu(r)\,dr+1\\
		&=\int_0^{|z|}\frac{1}{\widehat{\om}(t)^p(1-t)^{(1-N)p}}\widehat{\nu}\left(\frac{t}{|z|}\right)\,dt+1\\
    &\le\int_0^{|z|}\frac{\widehat{\nu}(t)}{\widehat{\om}(t)^p(1-t)^{(1-N)p}}\,dt+1,\quad z\in\D.
    \end{split}
    \end{equation*}
Thus \eqref{eq:kernelmix-generalell} is satisfied and the proof of (i) is complete.

The assertion \eqref{eq:kernelmix-generalX} is an immediate consequence of \eqref{eq:kernelmix-general}, the trivial inequality $|1-\overline{z}\zeta|\ge1-|\zeta|$, the fact $\nu_{[Np]}\in\DDD$, which follows from Lemma~\ref{room}(ii), and \cite[Theorem~1(ii)]{PR2016/1}. We underline here that this argument does not work for $\nu\in\DD$ in general, because $\nu_{[Np]}$ does not necessarily belong to $\DDD$ when $\nu\in\DD$ by \cite[Theorem~3]{PR2020}.      
\end{proof}

\section{Proofs of the main results}\label{s3}

In this section we first build up the rest of the auxiliary results needed for proving Theorems~\ref{Theorem:hankel-characterization} and~\ref{proposition:Bloch-necessary} stated in the introduction and then prove the theorems themselves. We begin with showing that, under appropriate hypotheses on $f$ and the weights involved, the small Hankel operator induced by $\overline{f}$ only depends on its anti-analytic component $\overline{P_\om(f)}$.

\begin{lemma}\label{lemma:hankel-P}
Let $1<p<\infty$, $\om\in\DD$ and $f\in L^1_{\om_{\log}}$. Then $h^\om_{\overline{f}}(g)=h^\om_{\overline{P_\om(f)}}(g)$ for all $g\in H^\infty$.
\end{lemma}

\begin{proof}
Tonelli's theorem and \cite[Theorem~1]{PR2016/1} yield	
	\begin{equation*}
	\begin{split}
	\left|h^\om_{\overline{P_\om(f)}}(g)(\xi)\right|
	&\le\|B_\xi^\om\|_{H^\infty}\|g\|_{H^\infty}\int_\D\left(\int_\D|f(\zeta)||B_z^\om(\zeta)|\om(\zeta)\,dA(\zeta)\right)\om(z)\,dA(z)\\
	&=\|B_\xi^\om\|_{H^\infty}\|g\|_{H^\infty}\int_\D|f(\zeta)|\left(\int_\D|B_\zeta^\om(z)|\om(z)\,dA(z)\right)\om(\zeta)\,dA(\zeta)\\
	&\lesssim\|B_\xi^\om\|_{H^\infty}\|g\|_{H^\infty}\|f\|_{L^1_{\om_{\log}}}<\infty,\quad \xi\in\D,
	\end{split}
	\end{equation*}
and therefore Fubini's theorem, and the fact that $g\in A^{p}_\om\subset A^1_\om$, yields
	\begin{equation}\label{suvlkjh}
	\begin{split}
	h^\om_{\overline{P_\om(f)}}(g)(\xi)
	&=\int_\D\left(\overline{\int_\D f(\zeta)\overline{B_z^\om(\zeta)}\om(\zeta)\,dA(\zeta)}\right)g(z)B_\xi^\om(z)\om(z)\,dA(z)\\
	&=\int_\D\overline{f(\zeta)}\left(\int_\D g(z)B_\xi^\om(z)\overline{B_\zeta^\om(z)}\om(z)\,dA(z)\right)\om(\zeta)\,dA(\zeta)\\
	&=\int_\D\overline{f(\zeta)}g(\zeta)B_\xi^\om(\zeta)\om(\zeta)\,dA(\zeta)
	=h^\om_{\overline{f}}(g)(\xi),\quad \xi\in\D.
	\end{split}
	\end{equation}
Thus $h^\om_{\overline{f}}(g)=h^\om_{\overline{P_\om(f)}}(g)$ for all $g\in H^\infty$ as claimed.
\end{proof}

The next lemma reveals that $\overline{P_\om V_{\om,\nu}(f)}=\overline{P_\om(f)}$, again under appropriate hypotheses on $f$ and the weights involved.

\begin{lemma}\label{lemma:V}
Let $\om$ be a radial weight and $\nu:\D\to[0,\infty)$ a radial function such that $\om\nu$ is a weight. Further, let $f:\D\to\C$ be measurable such that the function $z\mapsto f(z)\|B^{\om\nu}_z\|_{L^1_{\om\nu}}$ belongs to $L^1_\om$. Then $\overline{P_\om V_{\om,\nu}(f)}=\overline{P_\om(f)}$.
\end{lemma}

\begin{proof}	
Since the function $z\mapsto f(z)\|B^{\om\nu}_z\|_{L^1_{\om\nu}}$ belongs to $L^1_\om$ by the hypothesis, Fubini's theorem yields
	\begin{equation*}
	\begin{split}
	\overline{P_\om V_{\om,\nu}(f)(\zeta)}
	&=\int_\D\left(\nu(u)\int_\D\overline{f(v)}B^{\om\nu}_u(v)\om(v)\,dA(v)\right)B^\om_\zeta(u)\om(u)\,dA(u)\\
	&=\int_\D\overline{f(v)}\left(\int_\D B_\zeta^\om(u)\overline{B^{\om\nu}_v(u)}\om(u)\nu(u)\,dA(u)\right)\om(v)\,dA(v)\\
	&=\int_\D\overline{f(v)}B^\omega_\zeta(v)\om(v)\,dA(v)
    =\overline{P_\om(f)(\zeta)},\quad\zeta\in\D,
	\end{split}
	\end{equation*}
and thus the assertion is proved.
\end{proof}

With these preparations we can show that $V_{\om,\nu}(f)\in L^\infty$ is a sufficient condition for $h^\om_{\overline{f}}: A^p_\om\to\overline{A^p_\om}$ to be bounded. 

\begin{proposition}\label{Proposition:sufficient-V}
Let $1<p<\infty$, $\om\in\DDD$ and $f\in L^1_{\om_{\log}}$. Further, let $\nu:\D\to[0,\infty)$ be a radial function such that $\om\nu$ is a weight, and assume that the function $z\mapsto f(z)\|B^{\om\nu}_z\|_{L^1_{\om\nu}}$ belongs to $L^1_\om$.
If $V_{\om,\nu}(f)\in L^\infty$, then $h^\om_{\overline{f}}: A^p_\om\to 
\overline{A^p_\om}$ is bounded and
	$$
	\|h^\om_{\overline{f}}\|_{A^p_\om\to \overline{A^p_\om}}\lesssim\|V_{\om,\nu}(f)\|_{L^\infty}.
	$$
\end{proposition}

\begin{proof}
By Lemmas~\ref{lemma:hankel-P} and \ref{lemma:V}, and \cite[Theorem~9]{PelRat2020} we have
	\begin{equation*}
	\begin{split}
	\|h^\om_{\overline{f}}(g)\|_{L^{p}_\om}
	&=\|h^\om_{\overline{P_\om(f)}}(g)\|_{L^{p}_\om}
	=\|h^\om_{\overline{P_\om V_{\om,\nu}(f)}}(g)\|_{L^{p}_\om}
	=\|h^\om_{\overline{V_{\om,\nu}(f)}}(g)\|_{L^{p}_\om}\\
	&\le\|P_\om^+(|V_{\om,\nu}(f)g|)\|_{L^{p}_\om}
	\lesssim\|V_{\om,\nu}(f)g\|_{L^{p}_\om}
	\le\|V_{\om,\nu}(f)\|_{L^\infty}\|g\|_{A^{p}_\om},\quad g\in H^\infty,
	\end{split}
	\end{equation*}
and the assertion follows from the BLT-theorem~\cite[Theorem~I.7]{ReedSimon}, because $H^\infty$ is dense in $A^p_\om$ since $\om$ is radial.
\end{proof}

The necessity of the condition $V_{\om,\nu}(f)\in L^\infty$, with $\nu(z)=(1-|z|)^n$ for some $n$ large enough, for $h^\om_{\overline{f}}: A^p_\om\to\overline{A^p_\om}$ to be bounded is established next.

\begin{proposition}\label{proposition:Bounded-necessary}
Let $1<p<\infty$, $\om\in\DDD$ and $f\in L^1_{\om_{\log}}$ such that $h^\om_{\overline{f}}:A^p_\om\to\overline{A^p_\om}$ is bounded. Then there exists $n_0=n_0(\omega)\in\N$ such that, for each $n\ge n_0$, the weight $\nu(z)=(1-|z|)^n$ satisfies $\|V_{\om,\nu}(f)\|_{L^\infty}\lesssim\|h^\om_{\overline{f}}\|_{A^p_\om\to\overline{A^p_\om}}$.
\end{proposition}

\begin{proof}
By Tonelli's theorem, \cite[Theorem~1]{PR2016/1} and the hypothesis $f\in L^1_{\om_{\log}}$, we deduce
	\begin{equation*}
	\begin{split}
	\int_\D\left(\int_\D|f(\zeta)||g(\zeta)||B_z^\om(\zeta)|\om(\zeta)\,dA(\zeta)\right)|h(z)|\om(z)\,dA(z)
	\lesssim\|g\|_{H^\infty}\|h\|_{H^\infty}\|f\|_{L^1_{\om_{\log}}}<\infty,
	\end{split}
	\end{equation*}
and hence Fubini's theorem yields
	\begin{equation}\label{eq:fubini}
	\begin{split}
	\left\langle \overline{h^\om_{\overline{f}}(g)},h\right\rangle_{A^2_\om}
	&=\int_\D\left(\overline{\int_\D\overline{f(\zeta)}g(\zeta)B_z^\om(\zeta)\om(\zeta)\,dA(\zeta)}\right)\overline{h(z)}\om(z)\,dA(z)\\
	&=\int_\D f(\zeta)\overline{g(\zeta)}\left(\int_\D \overline{B_z^\om(\zeta)h(z)}\om(z)\,dA(z)\right)\om(\zeta)\,dA(\zeta)\\
	&=\int_\D f(\zeta)\overline{g(\zeta)}\left(\overline{\int_\D h(z)\overline{B_\zeta^\om(z)}\om(z)\,dA(z)}\right)\om(\zeta)\,dA(\zeta)\\
	&=\int_\D f(\zeta)\overline{g(\zeta)h(\zeta)}\om(\zeta)\,dA(\zeta), \quad g\in H^\infty,\quad h\in H^\infty.
	\end{split}
	\end{equation}
Therefore H\"older's inequality and the boundedness of $h^\om_{\overline{f}}:A^p_\om\to\overline{A^p_\om}$ give
	\begin{equation}\label{jhf}
	\begin{split}
	\left|\int_\D f(\zeta)\overline{g(\zeta)h(\zeta)}\om(\zeta)\,dA(\zeta)\right|
	&=\left|\left\langle \overline{h^\om_{\overline{f}}(g)},h\right\rangle_{A^2_\om}\right|
	\le\left\|\overline{h^\om_{\overline{f}}(g)}\right\|_{A^{p}_\om}\|h\|_{A^{p'}_\om}\\
	&\le\|h^\om_{\overline{f}}\|_{A^p_\om\to\overline{A^p_\om}}\|g\|_{A^p_\om}\|h\|_{A^{p'}_\om}, \quad
	g\in H^\infty,\quad h\in H^\infty.
	\end{split}
	\end{equation}

If $p\ge 2$, then, for each $z\in\D$ and $n\in\N$, choose the functions $g,h\in H^\infty$ such that $g(\zeta)=B_z^{\om\nu}(\zeta)(1-\overline{z}\zeta)^n$ and $h(\zeta)=(1-\overline{z}\zeta)^{-n}$ for all $\zeta\in\D$. Then $gh=B^{\om\nu}_z$ on the left-hand side of \eqref{jhf}, and therefore
	\begin{equation*}
	\begin{split}
	\left|\int_\D f(\zeta)\overline{B^{\om\nu}_z}(\zeta)\om(\zeta)\,dA(\zeta)\right|
	&\le\|h^\om_{\overline{f}}\|_{A^p_\om\to\overline{A^p_\om}}
	\left(\int_\D|B^{\om\nu}_z(\zeta)(1-\overline{z}\z)^n|^p\om(\z)\,dA(\z)\right)^\frac1p\\
	&\quad\cdot\left(\int_\D\frac{\om(\zeta)}{|1-\overline{z}\zeta|^{np'}}\,dA(\zeta)\right)^\frac{1}{p'},\quad z\in\D.
	\end{split}
	\end{equation*}
Next, take $n_0=n_0(\omega)$ such that $n_0\ge\lambda+1$, where $\lambda=\lambda(\omega)\ge0$ is that of Lemma~\ref{Lemma:weights-in-D-hat}(iii). Then we have
	\begin{equation*}
	\int_\D\frac{\om(\zeta)}{|1-\overline{z}\zeta|^{np'}}\,dA(\zeta)
	\lesssim\frac{\widehat{\om}(z)}{(1-|z|)^{np'-1}},\quad z\in\D,
	\end{equation*}
for each fixed $n\ge n_0$. Moreover, since $\om\in\DDD$ by the hypothesis, $\om\nu=\om_{[n]}\in\DDD$ and $\widehat{\om_{[n]}}(z)\asymp\widehat{\om}(z)(1-|z|)^n$ for all $n\in\N$ by Lemma~\ref{room}(ii). Therefore Lemma~\ref{kernelmix-general} 
 yields
	\begin{equation*}
	\begin{split}
	\int_\D|B^{\om\nu}_z(\zeta)(1-\overline{z}\z)^n|^p\om(\z)\,dA(\z)
	&\lesssim\int_0^{|z|}\frac{\widehat{\om}(t)}{\widehat{\om_{[n]}}(t)^p(1-t)^{p(1-n)}}\,dt+1\\
	&\asymp\int_0^{|z|}\frac{dt}{\widehat{\om}(t)^{p-1}(1-t)^{p}}+1\\
	&\asymp\frac1{\left(\widehat{\om}(z)(1-|z|)\right)^{p-1}},\quad z\in\D.
	\end{split}
	\end{equation*}
Thus
	\begin{equation}\label{eq:j2}
	\begin{split}
	\sup_{z\in\D}(1-|z|)^n\left|\int_\D f(\zeta)\overline{B^{\om\nu}_z(\zeta)}\om(\zeta)\,dA(\zeta)\right|
	&\lesssim\|h^\om_{\overline{f}}\|_{A^p_\om\to\overline{A^p_\om}},
	\end{split}
	\end{equation}
and the assertion follows.

If $p'>2$, then, for each $z\in\D$ and $n\in\N$, choose the functions $h,g\in H^\infty$ such that $h(\zeta)=B_z^{\om\nu}(\zeta)(1-\overline{z}\zeta)^n$ and $g(\zeta)=(1-\overline{z}\zeta)^{-n}$ for all $\zeta\in\D$. Then $gh=B^{\om\nu}_z$ in \eqref{jhf}, and by arguing as above we obtain \eqref{eq:j2}. This finishes the proof of the proposition. 
\end{proof}

By using Propositions~\ref{Proposition:sufficient-V} and~\ref{proposition:Bounded-necessary} it is now easy to prove Theorem~\ref{Theorem:hankel-characterization}, stated in the introduction.

\medskip

\begin{Prf}{\em{Theorem~\ref{Theorem:hankel-characterization}.}}
If $n\in\N$, then
	$$
	\|B^{\om_{[n]}}_z\|_{L^1_{\om_{[n]}}}\asymp\log\frac{e}{1-|z|},\quad z\in\D,
	$$
by Lemma~\ref{room}(ii) and \cite[Theorem~1]{PR2016/1}. Therefore 
	\begin{equation*}
	\begin{split}
	\left\|f(\cdot)\left\|B^{\om_{[n]}}_{(\cdot)}\right\|_{L^1_{\om_{[n]}}}\right\|_{L^1_\om}
	&\asymp\|f\|_{L^1_{\om_{\log}}}<\infty
	\end{split}
	\end{equation*}
by the hypothesis $f\in L^1_{\om_{\log}}$. Consequently, the assertion of the theorem follows by Propositions~\ref{Proposition:sufficient-V} and~\ref{proposition:Bounded-necessary}.
\end{Prf}

\medskip

We now proceed towards the proof of Theorem~\ref{proposition:Bloch-necessary}.  The next result about the fractional derivatives of the dilated functions $f_r(z)=f(rz)$ will be used in its proof.

\begin{lemma}\label{le:nuevo}
Let $\om, \nu$ be radial weights and $f\in\H(\D)$ such that  $V_{\om,\widehat{\nu}}(f)\in L^\infty$. Then
	$$
	\|V_{\om,\widehat{\nu}}(f_r)\|_{L^\infty}\le  \|V_{\om,\widehat{\nu}}(f)\|_{L^\infty}, \quad 0<r<1.
	$$ 
\end{lemma}

\begin{proof} Assume that $f\not\equiv0$, otherwise the result is trivial.
Since $z\mapsto f(z)=\sum_{n=0}^\infty\widehat{f}(n)z^n$ belongs to $\H(\D)$, and 
	$$
	\frac{\om_{2n+1}}{(\om\widehat{\nu})_{2n+1}}
	\le\frac{\om_1}{\e^{n+\frac{1}{2}}\widehat{\om\widehat{\nu}}(\sqrt{\e})},
	\quad 0<\e<1,\quad n\in \N\cup\{0\},
	$$ 
we deduce that $z\mapsto\frac{V_{\om,\widehat{\nu}}(f)(z)}{\widehat{\nu}(z)}=\sum_{n=0}^\infty \frac{\widehat{f}(n) \om_{2n+1}}{(\om\widehat{\nu})_{2n+1}}z^n$ belongs to $\H(\D)$.  Moreover, an analogous argument shows that
	\begin{equation}\label{eq:nuevo1}
	\frac{V_{\om,\widehat{\nu}}(f_r)}{\widehat{\nu}}=\left( \frac{V_{\om,\widehat{\nu}}(f)}{\widehat{\nu}}\right)_r
	\end{equation} 
is analytic in $D\left(0,\frac{1}{r}\right)$ for each $0<r<1$. By setting $\widehat{\nu}\equiv0$ on $\partial\D$, the function $\widehat{\nu}$ becomes continuous $\overline{\D}$. Hence we deduce that $V_{\om,\widehat{\nu}}(f)$ is continuous in $\D$, and $V_{\om,\widehat{\nu}}(f_r)$ is continuous in $\overline{\D}$ for each for each $0<r<1$. Consequently, for each $0<r<1$ there exists $z^r\in \D$ (observe that $V_{\om,\widehat{\nu}}(f_r)(\zeta)=0$ for each $\zeta\in \partial\D$) such that 
$\|V_{\om,\widehat{\nu}}(f_r)\|_{L^\infty}=\left| V_{\om,\widehat{\nu}}(f_r)(z^r)\right|$. 
Therefore \eqref{eq:nuevo1} and the maximum modulus principle imply
	\begin{equation*}
	\begin{split}
	\|V_{\om,\widehat{\nu}}(f_r)\|_{L^\infty}
	&=\widehat{\nu}(z^r)\left| \frac{V_{\om,\widehat{\nu}}(f_r)(z^r)}{\widehat{\nu}(z^r)}\right|
	=\widehat{\nu}(z^r)\left| \left( \frac{V_{\om,\widehat{\nu}}(f)}{\widehat{\nu}}\right)(rz^r)\right|\\
	&\le\widehat{\nu}(z^r)\max_{|u|\le |z^r|}\left| \left( \frac{V_{\om,\widehat{\nu}}(f)}{\widehat{\nu}}\right)(u)\right|
	=\widehat{\nu}(z^r)\max_{|u|
	=|z^r|}\left|\left(\frac{V_{\om,\widehat{\nu}}(f)}{\widehat{\nu}}\right)(u)\right|\\
	&=\max_{|u|= |z^r|}\left| V_{\om,\widehat{\nu}}(f)(u)\right|\le  \|V_{\om,\widehat{\nu}}(f)\|_{L^\infty}, \quad 0<r<1.
	\end{split}
	\end{equation*}
This finishes the proof.
\end{proof}

The following characterization of the Bloch space, which is interesting in its own right, is one of the key ingredients in the proof of Theorem~\ref{proposition:Bloch-necessary}. 

\begin{proposition}\label{pr:Blochdescription}
Let $f\in \H(\D)$ and $\omega,\nu\in\DDD$. Then $f\in\B$ if and only if $V_{\om,\widehat{\nu}}(f)\in L^\infty$. Moreover, 
	$$
	\|f\|_{\B}\asymp \|V_{\om,\widehat{\nu}}(f)\|_{L^\infty}.
	$$
\end{proposition}

\begin{proof}
Let first $f\in\B$. Then, by \cite[Theorem~3]{PelRat2020}, there exists $h\in L^\infty$ such that $P_\omega(h)=f$ and 
$\| h\|_{L^\infty}\asymp\|f\|_{\B}$. This together with Fubini's theorem, \cite[Theorem~1]{PR2016/1} and Lemma~\ref{room}(iii) yields
	$$
	\int_{\D} f(\zeta)\overline{B^{\om\widehat{\nu}}_z(\zeta)}\omega(\zeta)\,dA(\zeta)
	=\int_{\D}h(u)\overline{B^{\om\widehat{\nu}}_z(u)}\omega(u)\,dA(u),\quad z\in\D.
	$$
Therefore, by Lemma~\ref{room}(ii) and \cite[Theorem~1]{PR2016/1}, we deduce
	\begin{equation*}
	\begin{split}
	\left|\int_{\D} f(\zeta)\overline{B^{\om\widehat{\nu}}_z(\zeta)}\omega(\zeta)\,dA(\zeta) \right|
	&\le\|h\|_{L^\infty}\int_{\D}\left|B^{\om\widehat{\nu}}_z(u)\right|\omega(u)\,dA(u)\\
	&\lesssim \|f\|_{\B}\left(\int_{0}^{|z|} \frac{\widehat{\omega}(t)}{\widehat{\omega\widehat{\nu}}(t)(1-t)}\,dt+1 \right)\\
	&\asymp \|f\|_{\B}\left(\int_{0}^{|z|} \frac{dt}{\widehat{\nu}(t)(1-t)}+1\right),\quad z\in\D.
	\end{split}
	\end{equation*}
Since it is well known that $\nu\in\Dd$ if and only if there exists $\alpha=\alpha(\nu)>0$ such that 
	$$
	\frac{\widehat{\nu}(r)}{(1-r)^\alpha}\lesssim\frac{\widehat{\nu}(t)}{(1-t)^\alpha},\quad 0\le t\le r<1,
	$$
we deduce
	$$
	\int_{0}^{|z|} \frac{dt}{\widehat{\nu}(t)(1-t)}
	\lesssim\frac{(1-|z|)^\alpha}{\widehat{\nu}(z)}\int_{0}^{|z|}\frac{dt}{(1-t)^{1+\alpha}}
	\lesssim\frac1{\widehat{\nu}(z)},\quad z\in\D,
	$$
and it follows that $\|V_{\om,\widehat{\nu}}(f)\|_{L^\infty}\lesssim\|f\|_{\B}$ for each fixed $\om,\nu\in\DDD$. 

Conversely, assume that $\|V_{\om,\widehat{\nu}}(f)\|_{L^\infty}<\infty$. Let $k\in A^1_\omega$. Then the reproducing formula for functions in $A^1_{\omega\widehat{\nu}}$ and Fubini's theorem give
	\begin{equation*}
	\begin{split}
	\left|\left\langle k,f_r\right\rangle_{A^2_\om}\right|
	&=\left|\left\langle P_{\om\widehat{\nu}}k,f_r\right\rangle_{A^2_\om}\right|
	=\left| \int_{\D} k(u)\overline{V_{\om,\widehat{\nu}}(f_r)} \omega(u)\,du\right|
	\le\|k\|_{A^1_\omega}\|V_{\om,\widehat{\nu}}(f_r)\|_{L^\infty},\quad 0<r<1.
	\end{split}
	\end{equation*}
Moreover, by Lemma~\ref{le:nuevo}, we have $\|V_{\om,\widehat{\nu}}(f_r)\|_{L^\infty}\le\|V_{\om,\widehat{\nu}}(f)\|_{L^\infty}$, $0<r<1$, and hence
 	$$
	\left|\left\langle k,f\right\rangle_{A^2_\om}\right|\le \|k\|_{A^1_\omega} \|V_{\om,\widehat{\nu}}(f)\|_{L^\infty},\quad k\in A^1_\omega.
	$$
But $\om\in\DDD$ by the hypothesis, and therefore $(A^1_\om)^\star$ is isomorphic to the Bloch space via the $A^2_\omega$-pairing with equivalence of norms~\cite[Theorem~3]{PelRat2020}. Hence $f\in\B$ with
	$$
	\|f\|_{\B}\lesssim\|V_{\om,\widehat{\nu}}(f)\|_{L^\infty}
	$$
for each fixed $\om,\nu\in\DDD$. This finishes the proof of the proposition.
\end{proof}

With these preparations we can now prove Theorem~\ref{proposition:Bloch-necessary}.

\medskip

\begin{Prf}{\em{Theorem~\ref{proposition:Bloch-necessary}.}} Let first $f\in\B$. Then Proposition~\ref{pr:Blochdescription} yields 
	$$
	\|V_{\om,\nu}(f)\|_{L^\infty}\lesssim \|f\|_{\B}
	$$ 
for each $\nu(z)=(1-|z|)^n$ with $n\in\N$. This together with the obvious embedding $\B\subset A^1_{\omega_{\log}}$ and Theorem~\ref{Theorem:hankel-characterization} shows that $h^\om_{\overline{f}}:A^p_\om\to\overline{A^p_\om}$ is bounded and $\|h^\om_{\overline{f}}\|_{A^p_\om\to\overline{A^p_\om}}\lesssim \|f\|_{\B}$.

Conversely, assume that $h^\om_{\overline{f}}:A^p_\om\to\overline{A^p_\om}$ is bounded. Then $\overline{h^\om_{\overline{f}}(1)}=f\in A^p_\omega\subset A^1_{\om_{\log}}$. Therefore, by Proposition~\ref{proposition:Bounded-necessary}, there exists $n_0=n_0(\omega)\in\N$ such that, for each $n\ge n_0$, the weight $\nu(z)=(1-|z|)^n$ satisfies $\|V_{\om,\nu}(f)\|_{L^\infty}\lesssim\|h^\om_{\overline{f}}\|_{A^p_\om\to\overline{A^p_\om}}$. This together with Proposition~\ref{pr:Blochdescription} implies $f\in\B$ with
	$$
	\|f\|_{\B}\lesssim \|V_{\om,\nu}(f)\|_{L^\infty}
	\lesssim\|h^\om_{\overline{f}}\|_{A^p_\om\to\overline{A^p_\om}}
	$$
for each fixed $n\ge n_0$. Thus the theorem is proved.
\end{Prf}

\section{Boundedness of $P_\om:\BMO(\Delta)_{\om,p}\to\B$}\label{s4}

In this section we will use Theorems~\ref{Theorem:hankel-characterization} and \ref{proposition:Bloch-necessary} to sketch of a new proof of the main theorem of the recent paper \cite{PRMZ2023}. Prior to state this result some definitions are needed. 
Let $\b(z,\z)$ denote the hyperbolic distance between the points $z$ and $\z$ in $\D$, and let $\Delta(z,r)$ stand for the hyperbolic disc of center $z\in\D$ and radius $0<r<\infty$.
 Further, let $\omega$ be a radial weight and $0<r<\infty$ such that $\omega\left(\Delta(z,r)\right)>0$ for all $z\in\D$. Then, for $f\in L^p_{\om,{\rm loc}}$ with $1\le p<\infty$, write 
    $$
    \MO_{\om,p,r}(f)(z)
		=\left(\frac{1}{\om(\Delta(z,r))}
		\int_{\Delta(z,r)}|f(\z)-\widehat{f}_{r,\om}(z)|^p\om(\z)\,dA(\z)\right)^{\frac{1}{p}},
    $$
where 
	$$
	\widehat{f}_{r,\om}(z)=\frac{\int_{\Delta(z,r)}f(\z)\om(\z)\,dA(\z)}{\om(\Delta(z,r))},\quad z\in\D.
	$$ 
The space $\BMO(\Delta)_{\om,p,r}$ consists of $f\in L^p_{\om}$ such that
    $$
    \|f\|_{\BMO(\Delta)_{\om,p,r}}=\sup_{z\in\D}\MO_{\om,p,r}(f)(z)<\infty.
    $$
It is known by \cite[Theorem~11]{PPR2020} that for each $\om\in\DDD$ there exists $r_0=r_0(\om)>0$ such that
	\begin{equation}\label{eq:intro1}
	\BMO(\Delta)_{\om,p,r}=\BMO(\Delta)_{\om,p,r_0}, \quad r\ge r_0.
	\end{equation}
We call this space $\BMO(\Delta)_{\om,p}$ and assume that the norm is always calculated with respect to a fixed $r\ge r_0$.

We write $\om\in\M$ if there exist constants $C=C(\om)>1$ and $K=K(\om)>1$ such that $\om_{x}\ge C\om_{Kx}$ for all $x\ge1$. It is known that $\Dd\subsetneq\M$ and $\DDD=\DD\cap\M$ by \cite[Proof of Theorem~3 and Proposition~14]{PelRat2020}.

The main result in \cite{PRMZ2023} reads as follows.

\begin{lettertheorem}\label{th:main}
Let $1<p<\infty$ and $\om\in\M$. Then the following statements are equivalent:
	\begin{itemize}
	\item[\rm(i)] There exists $r_0=r_0(\omega)\in(0,\infty)$ such that $\BMO(\Delta)_{\om,p,r}$ does not depend on $r$, provided $r\ge r_0$. Moreover, $P_\om:\BMO(\Delta)_{\om,p}\to\B$ is bounded;
	\item[\rm(ii)] $P_\om:L^\infty\to\B$ is bounded;
	\item[\rm(iii)] $\om\in\DD$.
	\end{itemize}
\end{lettertheorem}

We now show how this result can be obtained by using Theorems~\ref{Theorem:hankel-characterization} and \ref{proposition:Bloch-necessary}. Bearing in mind the continuous embedding $L^\infty\subset\BMO(\Delta)_{\omega,p}$, the comments presiding the statement and \cite[Theorem~1]{PelRat2020}, we only need to prove that $P_\om:\BMO(\Delta)_{\om,p}\to\B$ is bounded when $\om\in\DDD$. To see this, let $f\in\BMO_{\om,p}(\Delta)$. Then, by \cite[Theorem~11(ii)]{PPR2020} and its proof, $f$ can be represented in the form $f_1+f_2$, where $f_1\in\BA(\Delta)_{\om,p}$ and $f_2\in\BO(\Delta)$ such that $\|f_1\|_{\BA(\Delta)_{\om,p}}+\|f_2\|_{\BO(\Delta)}\lesssim \|f\|_{\BMO(\Delta)_{\om,p}}$. Here, $\BO(\Delta)$ is the space of Lipschitz continuous functions in the hyperbolic metric, and $f_1\in \BA(\Delta)_{\om,p}$ if the multiplication operator $M_f(g)(z)=f(z)g(z)$ is bounded from $A^p_\omega$ to $L^p_\omega$. Let us denote $\nu(z)=(1-|z|)^n$ for all $z\in\D$. If $f_1\in\BA(\Delta)_{\om,p}$, H\"older's inequality, \cite[Theorem~1]{PR2016/1} and Lemma~\ref{room}(ii) imply
	\begin{equation}\label{eq:f1}
	\begin{split}
	\|V_{\om,\nu}(f_1)\|_{L^\infty}\lesssim \|f_1\|_{\BA(\Delta)_{\om,p}}.
	\end{split}
	\end{equation}
Moreover, the inequality
	$$
	|f_2(z)-f_2(\z)|
	\lesssim(1+\beta(z,\z))\|f_2\|_{\BO(\Delta)} 
	\lesssim\frac{|1-\overline{z}\z|^{2\ep}}{(1-|z|)^{\ep}(1-|\z|)^{\ep}}\|f\|_{\BO(\Delta)},\quad z,\zeta\in\D, \quad \ep>0,
	$$
H\"older's inequality, Lemma~\ref{kernelmix-general} and \cite[Theorem~1]{PR2016/1}
yield 
 \begin{equation}\label{eq:f2}
	\begin{split}
	\|V_{\om,\nu}(f_2)\|_{L^\infty}\lesssim \|f_2\|_{\BO(\Delta)}
	\end{split}
	\end{equation}
for each $n\ge n_0$, where $n_0=n_0(\omega)$. Now that $\BMO(\Delta)_{\om,p}\subset L^1_{\om_{\log}}$ for each $1<p<\infty$, we may combine Theorem~\ref{Theorem:hankel-characterization} with \eqref{eq:f1} and \eqref{eq:f2}, to deduce that $h_{\overline{f}}^\om: A^s_\om\to\overline{A^s_\om}$ is bounded for any $1<s<\infty$, and
	\begin{equation*}
	\begin{split}
	\|h_{\overline{f}}^\om\|_{A^s_\om\to\overline{A^s_\om}}
	&\lesssim \|V_{\om,\nu}(f)\|_{L^\infty}
	\le\|V_{\om,\nu}(f_1)\|_{L^\infty}+ \|V_{\om,\nu}(f_2)\|_{L^\infty}\\
	&\lesssim\|f_1\|_{\BA(\Delta)_{\om,p}}+\|f_2\|_{\BO(\Delta)}
	\lesssim\|f\|_{\BMO(\Delta)_{\om,p}}.
	\end{split}
	\end{equation*}
Therefore Theorem~\ref{proposition:Bloch-necessary} and Lemma~\ref{lemma:hankel-P} yield $P_\omega(f)\in\B$ and 
  $$
	\|P_\omega(f)\|_{\B}
	\asymp\|h_{\overline{P_\om(f)}}^\om\|_{A^{s}_\om\to\overline{A^{s}_\om}}
  =\|h_{\overline{f}}^\om\|_{A^{s}_\om\to\overline{A^{s}_\om}}
	\lesssim\|f\|_{\BMO(\Delta)_{\om,p}}.
	$$
Hence $P_\om:\BMO(\Delta)_{\om,p}\to\B$ is bounded.

\bibliographystyle{amsplain}

\begin{thebibliography}{00}

\bibitem{Aleman:Constantin2004} 						A. ~Aleman and  O.~Constantin, 
																						Hankel operators on Bergman spaces and similarity to contractions, 
																						Int. Math. Res. Not. (35) (2004), 1785--1801.

\bibitem{AlCo} 															A.~Aleman and O.~Constantin, 
																						Spectra of integration operators on weighted Bergman spaces, 
																						J. Anal. Math. 109 (2009), 199--231.

\bibitem{BDTW2021} 													D. B\'ekoll\'e, H.~O.~Defo, E.~L. Tchoundja and B.~D. Wick, 
																						Little Hankel operators between vector-valued Bergman spaces on the unit ball, 
																						Integral Equ. Oper. Theory 93(3) (2021), Paper No. 28, 46 pp.

\bibitem{BL2005} 														A.~Bonami and L.~Luo, 
																						On Hankel operators between Bergman spaces on the unit ball, 
																						Houst. J. Math. 31(3) (2005), 815--827.

\bibitem{BoudreauxJGA19} 										B.~J.~Boudreaux,
																						Equivalent Bergman spaces with inequivalent weights,
																						J. Geom. Anal. 29 (2019), no. 1, 217--223.

\bibitem{DRWW2022}													Y. Duan, J. R\"atty\"a, S. Wang and F. Wu,
																						Two weight inequality for Hankel form on weighted Bergman spaces induced by doubling weights, 
																						Adv. Math.~431 (2023), Paper No. 109249, 47 pp.

\bibitem{D}          												P.~Duren, 
																						Theory of $H^p$ Spaces, 
																						Academic Press, 1970.

\bibitem{JPR87}															S.~Janson, 
																						J.~Peetre and R.~ Rochberg, 
																						Hankel forms and the Fock space, Rev. Mat. Iberoamericana 3 (1987), no.1, 61--138.

\bibitem{KR}           											T.~Korhonen and J. R\"atty\"a,
																						Zero sequences, factorization and sampling measures for weighted Bergman spaces,
																						Math. Z. 291 (2019), 1145--1173.

\bibitem{Ni}  															N.~Nikolski, 
																						Toeplitz Matrices and Operators, 
																						Translated from the French edition by Dani\`ele Gibbons and Greg Gibbons, 
																						Cambridge Studies in Advanced Mathematics, vol.182, Cambridge University Press, Cambridge, 2020.

\bibitem{Oliver} 														R.~V.~Oliver, 
																						Hankel operators on vector-valued Bergman spaces, 
																						Programa de Doctorat en Matem\`a-tiques, 
																						Departament de Matem\`atica i Inform\`atica, Universitat de Barcelona, 2017.

\bibitem{PZ2015} 														J.~ Pau, R.~ Zhao, 
																						Weak factorization and Hankel forms for weighted Bergman spaces on the unit ball, 
																						Math. Ann. 363 (1--2) (2015), 363--383.

\bibitem{PelSum14}      										J.~A.~ Pel\'aez,
																						Small weighted Bergman spaces,
																						Proceedings of the summer school in complex and harmonic analysis, and related topics, (2016).

\bibitem{PPR2020}														J. A. Pel\'aez, A. Per\"al\"a and J. R\"atty\"a,
																						Hankel operators induced by radial Bekoll\'e-Bonami weights on Bergman spaces,
																						Math. Z. 296 (2020), no. 1--2, 211--238.

\bibitem{PR}            										J.~A. Pel\'aez and J.~R\"atty\"a, 
																						Weighted Bergman spaces induced by rapidly increasing weights,
																						Mem. Amer. Math. Soc. 227 (2014).

\bibitem{PR2016/1}      			J. A. Pel\'aez and J. R\"atty\"a,
															Two weight inequality for Bergman projection,
															J. Math. Pures Appl. 105 (2016), 102--130.

\bibitem{PR2020}        			J. A. Pel\'aez and J. R\"atty\"a,
															Harmonic conjugates on Bergman spaces induced by doubling weights,
															Anal. Math. Phys. 10 (2020), no. 2, Paper No. 18, 22 pp.

\bibitem{PelRat2020}    			J.~A.~ Pel\'aez and J. R\"atty\"a,
															Bergman projection induced by radial weight,
															Adv. Math. 391 (2021). Paper No. 107950, 70 pp.

\bibitem{PRMZ2023} 			J.~A.~ Pel\'aez and J. R\"atty\"a, 
												Bergman projection and BMO in hyperbolic metric: improvement of classical result,
												Math. Z. 305 (2023), no. 2, Paper No. 19, 9 pp.

\bibitem{Perala} 				A. Per\"al\"a, 
												Vanishing Bergman kernels on the disk, 
												J. Geom. Anal. 28 (2) (2018) 1716--1727.
												
\bibitem{ReedSimon}  		M. Reed and B. Simon,
												Methods of modern mathematical physics. I: Functional analysis,
												Academic Press, New York, 1972.

\bibitem{Yamaji}				S.~Yamaji,
												Essential norm estimates for little Hankel operators on weighted Bergman spaces of the unit ball,
												Complex Anal. Oper. Theory 8 (2014), no.4, 863--873.

\bibitem{ZhuTAMS91} 		K.~Zhu,
												Hankel operators on the Bergman space of bounded symmetric domains,
												Trans. Amer. Math. Soc. 324 (1991), no.2, 707--730.

\bibitem{Zhu}           K.~Zhu,
                        Operator Theory in Function Spaces,
                        Second Edition, Math. Surveys and Monographs, Vol. 138, American Mathematical
                        Society: Providence, Rhode Island, 2007.

\end{thebibliography}

\end{document}